\documentclass{amsart}
\usepackage{amsfonts}

\newtheorem{theorem}{Theorem}
\theoremstyle{plain}

\newtheorem{lemma}{Lemma}

\numberwithin{equation}{section}
\input{tcilatex}

\begin{document}
\title{A remark on the Heat Equation and minimal Morse Functions on Tori and
Spheres}
\author{C. Cadavid$^{a}$}
\address{a \ Corresponding Author: Carlos Cadavid, {\small {Universidad
EAFIT, Departamento de Ciencias B\'{a}sicas, Bloque 38, Office 417 }}Carrera
49 No. 7 Sur -50, Medell\'{\i}n, Colombia, Phone: (57)(4)-2619500 Ext 9790,
Fax:(57)(4) 2664284}
\email{ccadavid@eafit.edu.co.}
\author{J. D. V\'{e}lez$^b$}
\address{b \ Juan D. V\'{e}lez Universidad Nacional, Medell\'{\i}n Colombia
S.A }
\email{jdvelez14@gmail.com}
\thanks{MSC 58E05, 35K05}

\begin{abstract}
Let $(M,g)$ be a compact, connected riemannian manifold that is homogeneous,
i.e. each pair of points $p,q\in M$ have isometric neighborhoods. This paper
is a first step towards an understanding of the extent to which it is true
that for each \textquotedblleft generic" initial condition $f_{0}$, the
solution to $\partial f/\partial t=\Delta _{g}f,f(\cdot ,0)=f_{0}$ is such
that for sufficiently large $t$, $f(\cdot ,t)$ is a minimal Morse function,
i.e., a Morse function whose total number of critical points is the minimal
possible on $M$. In this paper we show that this is true for flat tori and
round spheres in all dimensions.
\end{abstract}

\maketitle

\section{Introduction}

The Heat Equation is arguably the most important partial differential
equation in mathematics and physics. This equation also seems to be
ubiquitous in geometry as well as in many other branches of mathematics. In
a suggestive article \cite{La}, Serge Lang and Jay Jorgenson called the Heat
Kernel "... \textit{a universal gadget which is a dominant factor
practically everywhere in mathematics, also in physics, and has very simple
and powerful properties}". In this short note we present evidence
corroborating the relevance of the Heat Equation, in a seemingly novel
direction. Specifically, we present two concrete cases in which the Heat
Equation \textquotedblleft discovers all by itself" minimal Morse functions.
This could be, however, a manifestation of a phenomenon that might hold on
more general homogeneous riemannian manifolds.\newline

Let $M$ be a closed, connected, oriented smooth manifold, and let $g$ be a
riemannian metric on $M$. On each tangent space $T_{p}(M)$, the metric $g$
determines a bilinear function $\langle \ ,\rangle _{g}$. For any given
smooth function $f$ on $M$, let us recall that the gradient is defined as
the vector field $\mathrm{grad}(f)$ in $T(M)$ that satisfies $\langle 
\mathrm{grad}(f),\zeta \rangle _{g}=\zeta (f)$, for all $\zeta \in T(M).$
Let us denote by $\nabla $ the Levi-Civita connection determined by the
metric $g$. For each smooth vector field $X$ on $M$, its divergence is
defined as $\mathrm{div}(X)=\mathrm{trace}(\zeta \rightarrow \nabla _{\zeta
}(X))$. The Laplacian (or Laplace-Beltrami operator) on $(M,g)$ of a smooth
function $f:M\rightarrow \mathbb{R}$ is defined as $\Delta _{g}f=\mathrm{div}%
(\mathrm{grad}(f))$.

The Heat Equation on $(M,g)$ is the partial differential equation $\partial
f/\partial t=\Delta _{g}f\,.$ A solution to the initial condition problem 
\begin{equation}
\left\{ 
\begin{array}{ccl}
\partial f/\partial t & = & \Delta _{g}f \\ 
f(\cdot ,0) & = & f_{0}\in L^{2}(M)%
\end{array}%
\right.  \label{pci}
\end{equation}%
is a continuos function $f:M\times (0,\infty )\rightarrow \mathbb{R}$ such
that

\begin{enumerate}
\item For each fixed $t>0$, $f(\cdot ,t)$ is $C^{2}$ function, and for each $%
x\in M$, $f(x,\cdot )$ is $C^{1}.$

\item $\partial f/\partial t=\Delta _{g}f,$ and 
\begin{equation*}
\lim_{t\rightarrow 0^{+}}\int_{M}f(x,t)\psi (x)dV(x)=\int_{M}f_{0}(x)\psi
(x)dV(x)\,,
\end{equation*}%
for all $\psi \in C^{\infty }(M)$.
\end{enumerate}

It is a non trivial fact that for each $t>0,$ $f_{t}=f(\cdot ,t)$ is smooth (%
\cite{Ch}, \cite{Ro}). It is well known that the solution to problem (\ref%
{pci}) can be obtained in the following way. First, it can be seen that the
eigenvalues of the operator $\Delta _{g}$, understood as those $\lambda \in 
\mathbb{R}$ \ for which $\Delta _{g}f+\lambda f=0$ for some smooth function $%
f$ nonidentically zero, are nonnegative and form a discrete set (\cite{Ch}, 
\cite{Ro}), $\lambda _{0}=0<\lambda _{1}<\cdots <\lambda _{j}<\cdots $.
Moreover, for each $j\geq 0$, the corresponding eigenspace $E_{j}$ has
finite dimension $m_{j}$ and $E_{j}\subset C^{\infty }(M)$. By the
assumptions about $M$, $E_{0}$ is the one dimensional vector space of
constants functions. For each $j\geq 0$ let $B_{j}=\{\phi
_{j,i}:i=1,...,m_{j}\}$ be an orthonormal basis for $E_{j}$. Their union $%
B=\cup B_{j}$ is an orthonormal basis for $L^{2}(M)$. Then, the solution to
problem (\ref{pci}) can be written as 
\begin{equation}
f=\sum_{j\geq 0}e^{-\lambda _{j}t}\sum_{i=1}^{m_{j}}\langle f_{0},\phi
_{j,i}\rangle _{L^{2}(M)}\phi _{j,i}\,,  \label{solucioncalor}
\end{equation}%
where $\langle -,-\rangle _{L^{2}(M)}$ denotes the inner product of $L^{2}(M)
$.\newline

We want to address the following question: to what extent is it true that
for each``generic" initial condition $f_{0}$, the solution to (\ref{pci}) is
such that for sufficiently large $t$, $f_{t}$ is a minimal Morse function,
i.e., a Morse function whose total number of critical points is less or than
equal to that of any other Morse function on $M$? Below, we show that the
answer to this question is affirmative for flat tori and round spheres. A
key ingredient in the proof is the following result which is a corollary of 
\textit{Mather's Stability Theorem }\cite{Ma}.

\begin{theorem}[Stability of Morse functions]
\label{smf} Let $M$ de a smooth manifold and let $h:M\rightarrow \mathbb{R}$
be any Morse function. Then, there exists an open neighborhood $W_{h}$ in
the $C^{\infty }$ topology (of $C^{\infty }(M)$) such that any function $%
\varphi \in W_{h}$ is Morse with the same number of critical points as $h$.
When $M$ is compact, the $C^{\infty }$ topology (of $C^{\infty }(M)$) is the
union of all $C^{r}$ topologies (of $C^{\infty }(M)$) \cite{Hi}, so in this
case the conclusion can be restated as follows: There exists $r>0$ and $%
\epsilon >0$ such that, if $\left\Vert \varphi -h\right\Vert _{r}<\epsilon $
for $\varphi \in C^{\infty }(M)$ and $\left\Vert \cdot \right\Vert _{r}$
being the $C^{r}$-norm, then $\varphi $ is also a Morse function with the
same number of critical points as $h$.
\end{theorem}

\section{Tori}

Let $T^{n}$ denote the $n$-dimensional flat torus obtained as the quotient
of $\mathbb{R}^{n}$ by the obvious action of $(2\pi \mathbb{Z})^{n}$ by
isometries. In this case the spectrum of the Laplacian is the set $\{\lambda
_{j}:j\geq 0\}$ where $\lambda _{j}$ is the $j$-th integer that can be
written as a sum of squares $k_{1}^{2}+\cdots +k_{n}^{2}$ for some
nonnegative integers $k_{1},\ldots ,k_{n}$ \cite{Gr}. Notice that $\lambda
_{0}=0$ (as it should) and $\lambda _{1}=1$. An orthonormal basis $B_{j}$
for $E_{j}$ is given by the functions in $T^{n}$ induced by the functions in 
$\mathbb{R}^{n}$ of the form 
\begin{eqnarray}
g_{1,\mathbf{k}} &=&\cos (k_{1}x_{1}+\cdots +k_{n}x_{n})  \label{gkl} \\
g_{2,\mathbf{k}} &=&\sin (k_{1}x_{1}+\cdots +k_{n}x_{n}),
\end{eqnarray}%
where $\mathbf{k}$ denotes a $n$-tuple $(k_{1},\ldots k_{n})$, such that $%
k_{1}^{2}+\cdots +k_{n}^{2}=\lambda _{j}$. The set $B=\cup _{j\geq 0}B_{j}$
is an orthonormal basis of $L^{2}(T^{n})$. In particular, $B_{1}$ is $\{\cos
(x_{1}),\sin (x_{1}),\ldots ,\cos (x_{n}),\sin (x_{n})\}.$

\begin{lemma}
\label{minimalMorse} The function induced on $T^{n}$ by $h=%
\sum_{k=1}^{n}(a_{k}\cos (x_{k})+b_{k}\sin (x_{k}))$ is a Morse function, if
and only if, $a_{k}^{2}+b_{k}^{2}\neq 0$ for every $k$. Moreover, if it is
Morse it is a minimal Morse function.
\end{lemma}

\begin{proof}
First, we observe that a smooth function on $T^{n}$ is Morse if and only if
its pullback is Morse. Now, each partial derivative $\partial h/\partial
x_{k}=-a_{k}\sin (x_{k})+b_{k}\cos (x_{k})$ can be rewritten in the form $%
A_{k}\sin (x_{k}+\theta _{k})=0$, where $A_{k}=\sqrt{a_{k}^{2}+b_{k}^{2}},$
and the $\theta _{k}$ are suitable constants. We notice that $h$ is Morse if
and only if whenever $(x)=(x_{1},\ldots ,x_{n})\in \mathbb{R}^{n}$ is a
solution to the system $A_{k}\sin (x_{k}+\theta _{k})=0, k=1,\ldots,n$, $(x)$
is \textit{not} a zero of the the determinant of the Hessian matrix, which
can be readily computed as%
\begin{equation*}
\det \mathrm{Hess}(h)=(-1)^{n}(a_{1}\cos (x_{1})+b_{1}\sin (x_{1}))\cdots
(a_{n}\cos (x_{n})+b_{n}\sin (x_{n}))\,.
\end{equation*}%
It is easy to see that this is the case if and only if $a_{k}^{2}+b_{k}^{2}
\neq 0,$ for each $k=1,\ldots ,n$. On the other hand, it is a general fact
that the sum of the betti numbers of a manifold $M$ is a lower bound for the
number of critical points of any Morse function on $M$. Using K\"{u}nneth's
formula we see that the sum of the betti numbers of $T^{n}$ equals $2^{n}$.
Finally, the system $A_{k}\sin (x_{k}+\theta _{k})=0$, $k=1,\ldots ,n,$
under the assumption that $a_{k}^{2}+b_{k}^{2}\neq 0$, for all $k$, has $%
2^{n}$ solutions in the box $[0,2\pi )^{n}$. This allows us to conclude that 
$h$ is a minimal Morse function.
\end{proof}

\begin{theorem}
\label{toros}There exists a set $S\subset C^{\infty }(T^{n})$, that is dense
and open in the $C^{\infty }$ topology, such that for any initial condition $%
f_{0}\in S$, if $f$ is corresponding solution to (\ref{pci}), then there
exists $T>0,$ depending on $f_{0}$, so that for each $t\geq T$, the function 
$f_{t}=f(\cdot ,t)$ is minimal Morse.
\end{theorem}

\begin{proof}
Let $f_{0}=h_{0}+h_{1}+\cdots $ where each $h_{j}$ is the projection of $%
f_{0}$ on $E_{j}$. Let us fix a nonnegative integer $r$, and let $\left\Vert 
{\cdot }\right\Vert _{r}$ be the $C^{r}$ norm on $C^{\infty }(M)$ (see \cite%
{Hi}). As noticed in the introduction, $f_{t}=\sum_{j=0}^{\infty
}e^{-\lambda _{j}t}h_{j}$. In order to prove that $f_{t}$ is minimal Morse
for all $t$ sufficiently large, it suffices to show that the same is true
for $e^{\lambda _{1}t}(f_{t}-h_{0})$. We have 
\begin{equation}
\left\Vert e^{\lambda _{1}t}(f_{t}-h_{0})-h_{1}\right\Vert _{r}=e^{-(\lambda
_{2}-\lambda _{1})t}\left\Vert \sum_{j=2}^{\infty }e^{-(\lambda _{j}-\lambda
_{2})t}h_{j}\right\Vert _{r}\text{ }\,.  \label{fourier}
\end{equation}%
Now we show that the second factor in the right hand side of (\ref{fourier})
is bounded on $T^{n}\times \lbrack 1,\infty )$. Let us write $%
h_{j}=\sum\nolimits_{\mathbf{k}}(c_{j,\mathbf{k}}g_{1,\mathbf{k}}+d_{j,%
\mathbf{k}}g_{2,\mathbf{k}})$, where the $g_{1\mathbf{,k}}$ and $g_{2,%
\mathbf{k}}$ are as in (\ref{gkl}). It can be seen inductively that for $%
i=1,2$, then $\partial g_{i,\mathbf{k}}/\partial x_{i_{s}}\cdots \partial
x_{i_{1}}$  is equal to $\pm k_{i_{1}}\cdots k_{i_{s}}g_{l,\mathbf{k}}$, for
some $l=1,2$. Using these we can estimate the sum in (\ref{fourier}) as
follows:%
\begin{eqnarray*}
\left\vert \partial h_{j}/\partial x_{i_{s}}\cdots \partial
x_{i_{1}}\right\vert  &\leq &\left\vert \sum\nolimits_{\mathbf{k}}c_{j,%
\mathbf{k}}\partial g_{1,\mathbf{k}}/\partial x_{i_{s}}\cdots \partial
x_{i_{1}}\right\vert +\left\vert \sum\nolimits_{\mathbf{k}}d_{j,\mathbf{k}%
}\partial g_{2,\mathbf{k}}/\partial x_{i_{s}}\cdots \partial
x_{i_{1}}\right\vert  \\
&\leq &\left( \sum\nolimits_{\mathbf{k}}c_{j,\mathbf{k}}^{2}\right)
^{1/2}\left( \sum\nolimits_{\mathbf{k}}\left( \partial g_{1,\mathbf{k}%
}/\partial x_{i_{s}}\cdots \partial x_{i_{1}}\right) ^{2}\right) ^{1/2} \\
&&+\left( \sum\nolimits_{\mathbf{k}}d_{j,\mathbf{k}}^{2}\right) ^{1/2}\left(
\sum\nolimits_{\mathbf{k}}\left( \partial g_{2,\mathbf{k}}/\partial
x_{i_{s}}\cdots \partial x_{i_{1}}\right) ^{2}\right) ^{1/2}
\end{eqnarray*}%
by the Cauchy-Schwartz inequality.\newline

\noindent Since $\partial g_{i,\mathbf{k}}/\partial x_{i_{s}}\cdots \partial
x_{i_{1}}=\pm k_{i_{1}}\cdots k_{i_{s}}g_{l,\mathbf{k}}$, for some $l=1,2$.
Now, the number of all possible $n$ tuples $\mathbf{k}$ $=(k_{1},\ldots
,k_{n})$ such that $k_{1}^{2}+\cdots +k_{n}^{2}=\lambda _{j}$ is clearly
bounded above by $\lambda _{j}^{n}.$ Thus, 
\begin{equation*}
\left\vert \partial g_{i,\mathbf{k}}/\partial x_{i_{s}}\cdots \partial
x_{i_{1}}\right\vert \leq \lambda _{j}^{(r+n)/2}\leq \lambda _{j}^{r+n}.
\end{equation*}%
Hence,%
\begin{equation*}
\left\vert \partial h_{j}/\partial x_{i_{s}}\cdots \partial
x_{i_{1}}\right\vert \leq \lambda _{j}^{r+n}\left( \left( \sum\nolimits_{%
\mathbf{k}}c_{j,\mathbf{k}}^{2}\right) ^{1/2}+\left( \sum\nolimits_{\mathbf{k%
}}d_{j,\mathbf{k}}^{2}\right) ^{1/2}\right) ,
\end{equation*}%
and consequently, $\left\vert \partial h_{j}/\partial x_{i_{s}}\cdots
\partial x_{i_{1}}\right\vert \leq 2\lambda _{j}^{r+n}\left\Vert {f_{0}}%
\right\Vert _{L^{2}(T^{n})}.$ It follows immediately that $\left\Vert
h_{j}\right\Vert _{r}\leq 2\lambda _{j}^{r+n}\left\Vert {f_{0}}\right\Vert
_{L^{2}(T^{n})}\,.$ From this we get the estimate

\begin{equation}
\left\Vert \sum_{j=2}^{\infty }e^{-(\lambda _{j}-\lambda
_{2})t}h_{j}\right\Vert _{r}\leq 2\left\Vert {f_{0}}\right\Vert
_{L^{2}(T^{n})}\sum_{j=2}^{\infty }\lambda _{j}^{r+n}e^{-(\lambda
_{j}-\lambda _{2})t}.  \label{ultima}
\end{equation}
It is to verify the convergence of the series on the right hand side of \ref%
{ultima} for each fixed value of $t\geq 1$. Since the sum of this series
decreases as $t$ increases, we deduce that $\Vert \sum_{j=2}^{\infty
}e^{-(\lambda _{j}-\lambda _{2})t}h_{j}\Vert _{r}$ is bounded on $%
T^{n}\times \lbrack 1,\infty ).$ Finally, this implies that

\begin{equation}  \label{limite}
\lim_{t\rightarrow \infty }\left\Vert e^{\lambda
_{1}t}(f_{t}-h_{0})-h_{1}\right\Vert _{r}=0
\end{equation}
for each $r\geq 0$.\newline

Let us assume that $h_{1}$ is a minimal Morse function. By Theorem \ref{smf}
there exist $r=r(h_{1})>0$ and $\epsilon =\epsilon (h_{1})>0$ such that
every $\varphi \in C^{\infty }(M)$ with $\Vert \varphi -h_{1}\Vert
_{r}<\epsilon _{0}$ is Morse and has the same number of critical points as $%
h_{1}$. Since $\lim_{t\rightarrow \infty }\left\Vert e^{\lambda
_{1}t}(f_{t}-h_{0})-h_{1}\right\Vert _{r}=0$, there is a $T_{0}>0$ such that
if $t\geq T_{0}$, 
\begin{equation*}
\left\Vert e^{\lambda _{1}t}(f_{t}-h_{0})-h_{1}\right\Vert _{r}<\epsilon\,.
\end{equation*}
This implies that if $t\geq T_{0}$, the function $e^{\lambda
_{1}t}(f_{t}-h_{0})-h_{1}$ is Morse and has the same number of critical
points as $h_{1}$. This allows us to conclude that for $t\geq T_{0}$, $%
e^{\lambda _{1}t}(f_{t}-h_{0})-h_{1}$ is minimal Morse, and so it is $f_{t}$%
. The proof is finished by verifying that the subset $S\subset C^{\infty
}(T^{n})$ of all functions whose projection on $E_{1}$ is minimal Morse is
open and dense in the $C^{\infty }$ topology. Since the $C^{0}$ norm bounds
from above the $L^{2}$ norm ($T^{n}$ is compact), then two functions close
in the $C^{0}$ sense are also close in the $L^{2}$ sense and therefore the
coefficients in their Fourier expansions are also close. A combination of
this with Lemma \ref{minimalMorse} shows that a smooth function that is
close to a smooth function whose projection on $E_{1}$ is minimal Morse,
must also have projection on $E_{1}$ that is minimal Morse. Finally, let us
show that $S$ is dense in $C^{\infty }(T^{n}).$ It suffices to prove that
for any smooth $f_{0}=h_{0}+h_{1}+\cdots $, $r>0$, and $\epsilon >0,$ there
is $\widetilde{f}_{0}\in S$ such that $\Vert f_{0}-\widetilde{f}_{0}\Vert
_{r}<\epsilon .$ It is enough to define $\widetilde{f}_{0}$ with the same
Fourier expansion as $f_{0}$ except that $h_{1}$ is replaced by a minimal
Morse function $\widetilde{h}_{1}$, such that $\Vert h_{1}-\widetilde{h}%
_{1}\Vert _{r}<\epsilon $.
\end{proof}

\section{Spheres}

In this section with prove a result similar to Theorem \ref{toros} for the $%
n $-dimensional sphere $S^{n}=\{(x_{1},\ldots ,x_{n+1}):x_{1}^{2}+\cdots
+x_{n+1}^{2}=1\}\subset \mathbb{R}^{n+1},n\geq 1,$ with the metric induced
from $\mathbb{R}^{n+1}.$ In this case, it is well known that the eigenvalues
of the Laplacian operator are given by $\lambda _{j}=j(j+n-1),$ $j\geq 0$
and the corresponding eigenfunctions, the so called \textit{spherical
harmonics}, are given as the restriction to the sphere of homogeneous
polynomials $H(x_{1},\ldots ,x_{n+1})$ in the coordinates of $\mathbb{R}%
^{n+1},$ of total degree $j,$ which satisfy the condition 
\begin{equation}
\partial ^{2}H/\partial x_{1}^{2}+\cdots +\partial ^{2}H/\partial
x_{n+1}^{2}=0 \,.  \label{laplacian}
\end{equation}%
The corresponding eigenspace of $\lambda _{j},$ $E_{_{j}},$ turns out to be
a space of dimension $\tbinom{n+j}{n}-\tbinom{n+j-2}{n}$ \cite{Ga1}.(In the
latter formula it is understood that $\tbinom{a}{b}=0$ in case $a<b$.) When $%
j=1,$ $\lambda _{1}=n,$ and a basis for $E_{_{1}}$ is given by the $n+1$
coordinate functions $x_{1},\ldots ,x_{n+1}.$ Each nonzero linear form in
these variables is a Morse function with two critical points, which is the
minimal possible, since this is precisely the sum of the betti numbers of $%
S^{n}$.\newline

\noindent Moreover, if $h$ denotes the restriction of a harmonic homogeneous
polynomial $H$ to $S^{n}$, the following estimate holds for the sup norm ($%
C^{0}$-norm in $M$):%
\begin{equation}
\left\Vert h\right\Vert _{0}\leq C_{n}\tbinom{n+j}{n}^{1/2}\left\Vert
h\right\Vert _{L^{2}(S^{n})}\,,  \label{estimativo}
\end{equation}%
where $C_{n}$ is a constant that only depends on $n$ (\cite{Ga1}, Section
7). Moreover, in the $C^{r}$-norm on $S^{n},$ the following more general
estimate holds (\cite{Ga2})%
\begin{equation}
\left\Vert h\right\Vert _{r}\leq C_{n}\tbinom{n+j}{n}^{1/2}(1+\lambda
_{j})^{r/2}\left\Vert h\right\Vert _{L^{2}(S^{n})}.  \label{estimativo cr}
\end{equation}%
With these preliminaries we can state the following theorem.

\begin{theorem}
There exists a set $S\subset C^{\infty }(S^{n})$, that is dense and open in
the $C^{\infty }$ topology, such that for any initial condition $f_{0}\in S$%
, if $f$ is corresponding solution to (\ref{pci}), then there exists $T>0,$
depending on $f_{0}$, so that for each $t\geq T$, the function $%
f_{t}=f(\cdot ,t)$ is minimal Morse.
\end{theorem}

\begin{proof}
Let $f_{0}=h_{0}+h_{1}+\cdots $ where each $h_{j}$ is the projection of $%
f_{0}$ on $E_{j}$, and let $f_{t}$ denote the solution to the Heat Equation
with initial condition $f_{0}$. In order to prove that $f_{t}$ is minimal
Morse for all $t$ sufficiently large, it suffices to show that the same is
true for $e^{\lambda _{1}t}(f_{t}-h_{0})$. Let us fix a positive integer $r.$
As in the previous proof, the key issue is to show that $\Vert
\sum_{j=2}^{\infty }e^{-(\lambda _{j}-\lambda _{2})t}h_{j}\Vert _{r}$ is
bounded on $S^{n}\times \lbrack 1,\infty ).$ Let us notice that $\lambda
_{j}-\lambda _{2}=j(j+n-1)-2(n+1)$ becomes greater that $n+j$ for any $j$
sufficiently large. On the other hand, $e^{x}$ is also greater than $x^{2}%
\tbinom{n+x}{n}(1+x)^{r}$ for all $x\geq N$. Thus, for all $j\geq N^{\prime
}\geq N$, and $t\geq 1$ 
\begin{equation*}
e^{-2(\lambda _{j}-\lambda _{2})t}\tbinom{n+j}{n}(1+j)^{r}<1/j^{2}\,.
\end{equation*}%
This implies that%
\begin{eqnarray*}
\left\Vert \sum_{j=N^{\prime }}^{m}e^{-(\lambda _{j}-\lambda
_{2})t}h_{j}\right\Vert _{r} &\leq &\sum_{j=N^{\prime }}^{m}e^{-(\lambda
_{j}-\lambda _{2})t}\left\Vert h_{j}\right\Vert _{r} \\
&\leq &C_{n}\sum_{j=N^{\prime }}^{m}e^{-(\lambda _{j}-\lambda _{2})t}\tbinom{%
n+j}{n}^{1/2}(1+\lambda _{j})^{r/2}\left\Vert h_{j}\right\Vert
_{L^{2}(S^{n})} \\
&\leq &C_{n}\left( \sum_{j=N^{\prime }}^{m}e^{-2(\lambda _{j}-\lambda _{2})t}%
\tbinom{n+j}{n}(1+\lambda _{j})^{r}\right) ^{1/2}\left( \sum_{j=N^{\prime
}}^{m}\left\Vert h_{j}\right\Vert _{L^{2}(S^{n})}^{2}\right) ^{1/2} \\
&\leq &C_{n}\left\Vert f_{0}\right\Vert _{L^{2}(S^{n})}\sum_{j=1}^{\infty
}1/j^{2}\,.
\end{eqnarray*}%
Hence, 
\begin{equation*}
\left\Vert \sum_{j=2}^{\infty }e^{-(\lambda _{j}-\lambda
_{2})t}h_{j}\right\Vert _{r}\leq K_{0}+C_{n}\left\Vert f_{0}\right\Vert
_{L^{2}(S^{n})}\sum_{j=1}^{\infty }1/j^{2}\,,
\end{equation*}%
where $K_{0}=\Vert \sum_{j=2}^{N^{\prime }-1}e^{-(\lambda _{j}-\lambda
_{2})}h_{j}\Vert _{r}$, which proves the claim. The rest of the proof is
exactly as that of the theorem above.
\end{proof}

\section{Acknowledgements}

We thank the Universidad Nacional of Colombia and Universidad Eafit for
their invaluable support.

\end{document}